\documentclass[12pt]{amsart}
\usepackage{amssymb}
\input amssym.def
\usepackage{pstricks}
\usepackage{url}
\newtheorem{theorem}{Theorem}
\newtheorem{lemma}[theorem]{Lemma}

\newtheorem{cor}[theorem]{Corollary}
\theoremstyle{definition}

\newcommand{\M}{\mathcal{M}}
\newcommand{\T}{\mathcal{T}}
\newcommand{\cO}{\mathcal{O}}
\newcommand{\cH}{\mathcal{H}}

\newcommand{\Z}{\mathbb{Z}}

\newcommand{\ab}{\mathrm{ab}}

\newcommand{\Sym}{\mathfrak{S}}

\newcommand{\Sp}{\mathrm{Sp}}

\numberwithin{equation}{section}
\numberwithin{theorem}{section}
\author{B{\l}a\.zej Szepietowski}
\title[On finite index subgroups of the mapping class group]{On finite index subgroups of the mapping class group of a nonorientable surface}
\address[]{Institute of Mathematics, Gda\'nsk University, Wita Stwosza 57,
80-952 Gda\'nsk, Poland} 
\email{blaszep@mat.ug.edu.pl}
\subjclass[2010]{20F38, 57N05.}
\keywords{Mapping class group, nonorinatble surface, finite index subgroup}
\thanks{Supported by NCN grant nr 2012/05/B/ST1/02171.}
\begin{document}
\maketitle

\begin{abstract}
Let $\M(N_{h,n})$ denote the mapping class group of a compact nonorientable surface of genus  $h\ge 7$ and $n\le 1$ boundary components, and let $\T(N_{h,n})$ be the subgroup of  $\M(N_{h,n})$ generated by all Dehn twists. It is known that $\T(N_{h,n})$ is the unique subgroup of $\M(N_{h,n})$ of index $2$. We prove that $\T(N_{h,n})$ (and also $\M(N_{h,n})$) contains a unique subgroup of index $2^{g-1}(2^g-1)$ up to conjugation, and  a unique subgroup of index $2^{g-1}(2^g+1)$ up to conjugation, where $g=\lfloor(h-1)/2\rfloor$. The other proper subgroups of $\T(N_{h,n})$ and $\M(N_{h,n})$ have index greater than  $2^{g-1}(2^g+1)$. In particular, the minimum index of a proper subgroup of $\T(N_{h,n})$ is $2^{g-1}(2^g-1)$.
\end{abstract}
\section{Introduction}
For a compact surface $F$, its {\it mapping class group} $\M(F)$ is the group of isotopy classes of all, orientation preserving if $F$ is orientable, homeomorphisms $F\to F$ equal to the identity on the boundary of $F$. A compact surface of genus $g$ with $n$ boundary components will be denoted by $S_{g,n}$ if it is orientable, or by $N_{g,n}$ if it is nonorientable.

It is well known that $\M(S_{g,n})$ is residually finite \cite{Gross}, and since $\M(N_{g,n})$ embeds in $\M(S_{g-1,2n})$ for $g+2n\ge 3$ \cite{BC,SzepB}, it is residually finite as well. It means that mapping class groups have a rich supply of finite index subgroups. On the other hand, Berrick, Gebhardt and Paris \cite{BGP} proved that for $g\ge 3$ the min­imum index of a proper subgroup of $\M(S_{g,n})$ is $m^-_g =2^{g-1}(2^g-1)$ (previously it was known  that the min­imum index is greater than $4g+4$ \cite{ParSmall}). 
More specifically, it is proved in \cite{BGP} that $\M(S_{g,n})$ contains a unique subgroup of index $m^-_g$ up to 
conjugation, a unique subgroup of index $m^+_g =2^{g-1}(2^g+1)$ up to 
conjugation, and all other proper subgroups of $\M(S_{g,n})$ have index 
strictly greater than $m^+_g$. The subgroups of indices $m^-_g$ and $m^+_g$ are 
constructed via the symplectic representation $\M(S_{g,n})\to\Sp(2g,\Z)$  
induced by the action of $\M(S_{g,n})$ on $H_1(S_{g,0},\Z)=\Z^{2g}$ (after gluing 
a disc along each boundary component of $S_{g,n}$). Passing mod $2$ we 
obtain an epimorphism $\theta_{g,n}\colon\M(S_{g,n})\to\Sp(2g,\Z_2)$. The orthogonal 
groups $O^-(2g,\Z_2)$ and $O^+(2g,\Z_2)$  are subgroups of $\Sp(2g,\Z_2)$ of indices respectively $m^-_g$ and $m^+_g$ (see \cite{BGP} and references there), and thus 
$\cO^-_{g,n}=\theta_{g,n}^{-1}(O^-(2g,\Z_2))$ and $\cO^+_{g,n}=\theta_{g,n}^{-1}(O^+(2g,\Z_2))$ are subgroups of 
$\M(S_{g,n})$ of indices respectively $m^-_g$ and $m^+_g$. 

Note that for $g\in\{1,2\}$ the minimum index of a proper subgroup of $\M(S_{g,n})$ is $2$. Indeed, the abelianization of $\M(S_{g,n})$ is $\Z_{12}$ for $(g,n)=(1,0)$, $\Z^n$ for $g=1$ and $n>0$, and $\Z_{10}$ for $g=2$ (see \cite{KorkSurv}). On the other hand, $\M(S_{g,n})$ is perfect for $g\ge 3$. Zimmermann \cite{Zimm} proved that for $g\in\{3,4\}$ the smallest nontrivial quotient of $\M(S_{g,0})$ is $\Sp(2g,\Z_2)$.  The problem of determining the smallest nontrivial quotient of $\M(S_{g,n})$ for $g\ge 5$ is open. On the other hand, Masbaum and Reid \cite{MR} proved that for fixed $g\ge 1$, every finite group occurs as a quotient of some finite index subgroup of $\M(S_{g,0})$.

\medskip

In this paper we consider the case of a nonorientable surface.
For $h\ge 2$ and $n\ge 0$, $\M(N_{h,n})$ contains a subgroup of  index $2$, namely the \emph{twist subgroup} $\T(N_{h,n})$ generated by all Dehn twists about two-sided curves \cite{Lick2,Stu_twist}. If $h\ge 7$, then $\T(N_{h,n})$ is perfect and equal to the commutator subgroup $[\M(N_{h,n}),\M(N_{h,n})]$ (see Theorem \ref{abel}). In particular, for $h\ge 7$, $\T(N_{h,n})$ is the unique subgroup of  $\M(N_{h,n})$ of index $2$. The motivating question for this paper is as follows.

\medskip

 \emph{What is the minimum index of a proper subgroup of $\T(N_{h,n})$?} 

\medskip

\noindent
To avoid complication, we restrict our attention to $n\le 1$. The reason is that for $n\le 1$, $\M(N_{h,n})$ and $\T(N_{h,n})$ have particularly simple generators. We emphasise, however, that finite generating sets for these groups are known for arbitrary $n$ \cite{Stu_twist, Stu_bdr}. It is worth mentioning at this point, that the first finite generating set for $\M(N_{h,0})$, $h\ge 3$, was obtained by Chillingworth \cite{Chill} using Lickorish's results \cite{Lick1,Lick2}.  

Our starting observation is that $\M(N_{h,n})$ and $\T(N_{h,n})$ admit epimorphisms onto $\Sp(2g,\Z_2)$, where $g =\lfloor(h - 1)/2\rfloor$, hence they contain subgroups of indices  $m^-_g$ and $m^+_g$. Here is the construction.
Set $V_h=H_1(N_{h,0},\Z_2)$. It is a $\Z_2$-module of rank $h$. It was proved by McCarthy and Pinkall \cite{McCP} that if $\varphi$ is an automorphism of $V_h$ which preserves the mod $2$ intersection pairing, then $\varphi$ is induced by a homeomorphism which is a product of Dehn twists. In other words the natural maps
$\M(N_{h,0})\to\mathrm{Aut}(V_h,\iota)$ and $\T(N_{h,0})\to\mathrm{Aut}(V_h,\iota)$ 
are epimorphisms, where $\iota$ is the mod $2$ intersection pairing on $V_h$ and 
$\mathrm{Aut}(V_h,\iota)$ is the group of automorphisms preserving $\iota$. By pre-composing these epimorphisms with those induced by gluing a disc along the boundary of $N_{h,1}$, respectively $\M(N_{h,1})\to\M(N_{h,0})$ and $\T(N_{h,1})\to\T(N_{h,0})$, we obtain epimorphisms from $\M(N_{h,1})$ and $\T(N_{h,1})$ onto $\mathrm{Aut}(V_h,\iota)$.
By \cite[Section 3]{KorkH1} and \cite[Lemma 8.1]{SzLow} we have isomorphisms 
\[\mathrm{Aut}(V_h,\iota)\cong\begin{cases}
\Sp(2g,\Z_2)&\textrm{if\ }h=2g+1\\
\Sp(2g,\Z_2)\ltimes\Z_2^{2g+1}&\textrm{if\ }h=2g+2
\end{cases}\]
In either case there is an epimorphism $\mathrm{Aut}(V_h,\iota)\to\Sp(2g,\Z_2)$. By pre-composing this epimorphism with the map from $\M(N_{h,n})$ (resp. $\T(N_{h,n})$) onto
$\mathrm{Aut}(V_h,\iota)$, we obtain for $n\in\{0,1\}$ epimorphisms 
\[\widetilde{\varepsilon}_{h,n}\colon\M(N_{h,n})\to\Sp(2g,\Z_2)\quad(\textrm{resp.\ }\varepsilon_{h,n}\colon\M(T_{h,n})\to\Sp(2g,\Z_2)).\]
Set 
\begin{align*}
&\widetilde{\cH}_{h,n}^-=\widetilde{\varepsilon}_{h,n}^{-1}(O^-(2g,\Z_2)),\quad\widetilde{\cH}_{h,n}^+=\widetilde{\varepsilon}_{h,n}^{-1}(O^+(2g,\Z_2))\\
&\cH_{h,n}^-=\varepsilon_{h,n}^{-1}(O^-(2g,\Z_2)),\quad \cH_{h,n}^+=\varepsilon_{h,n}^{-1}(O^+(2g,\Z_2))
\end{align*}
Here is our main result.
\begin{theorem}\label{ThmA}
Let $h=2g+r$ for $g\ge 3$, $r\in\{1,2\}$ and $n\in\{0,1\}$.
\begin{enumerate}
\item $\T(N_{h,n})$ is the unique subgroup of $\M(N_{h,n})$ of index $2$.
\item $\widetilde{\cH}_{h,n}^-$ (resp. $\cH_{h,n}^-$) is the unique subgroup of $\M(N_{h,n})$ (resp.  $\T(N_{h,n})$) of index $m_g^-$, up to conjugation.
\item $\widetilde{\cH}_{h,n}^+$ (resp. $\cH_{h,n}^+$) is the unique subgroup of $\M(N_{h,n})$ (resp.  $\T(N_{h,n})$) of index $m_g^+$, up to conjugation.
\item All other proper subgroups of $\M(N_{h,n})$ or $\T(N_{h,n})$ have index strictly greater then $m_g^+$,  and at least $5m_{g-1}^->m_g^+$ if $g\ge 4$.
\end{enumerate}
\end{theorem}
The nontrivial content of Theorem \ref{ThmA}
consists in points (2,3,4). The idea of the proof is as follows. Suppose 
that $H$ is a proper subgroup of $\T(N_{h,1})$ (for definiteness) of index $m\le m^+_g$. The group $\T(N_{h,1})$ contains an isomorphic copy of $\M(S_{g,1})$ and 
we prove in Lemma \ref{tran} that $H\cap \M(S_{g,1})$ has index $m$ in $\M(S_{g,1})$. 
Therefore, by \cite{BGP},  $H\cap \M(S_{g,1})$ is conjugate either to $\cO^-_{g,1}$ or to $\cO^+_{g,1}$. 
Then we prove in Theorem \ref{mainT} that $\cH^-_{h,1}$ (resp.  $\cH^+_{h,1}$) is the unique up to conjugacy 
subgroup of $\T(N_{h,1})$ of index $m^-_g$ (resp. $m^+_g$) such that $\cH^-_{h,1}\cap\M(S_{g,1})=\cO^-_{g,1}$ (resp.  $\cH^+_{h,1}\cap\M(S_{g,1})=\cO^+_{g,1}$ ). 

For $h\in\{5,6\}$ the abelianization of $\T(N_{h,n})$ is $\Z_2$, hence the minimum index of a proper subgroup is $2$. We prove in Theorem \ref{ThmB} that for $h\in\{5,6\}$ and $n\in\{0,1\}$, there are $4$ conjugacy classes of proper subgroups of $\T(N_{h,n})$ of index at most $m_2^+=10$. By \cite{Stu_twist}, the abelianization of $\T(N_{h,n})$ is $\Z_{12}$ for $(h,n)=(3,0)$, $\Z_{24}$ for $(h,n)=(3,1)$, and $\Z_2\times\Z$ for $(h,n)=(4,0)$. In particular, every positive integer occurs as an index of a subgroup of $\T(N_{4,0})$.

\section{Preliminaries}
\subsection{Permutations.}
Let $\Sym_m$ denote the full permutation group of 
$\{1,\dots, m\}$. The main tool used in this paper is the following well 
known relationship between index $m$ subgroups and maps to $\Sym_m$ . A 
group homomorphism $\varphi\colon G\to\Sym_m$ is transitive if the image acts transitively on $\{1,\dots, m\}$. If $\varphi\colon G\to\Sym_m$ is transitive, then 
$\mathrm{Stab}_\varphi(1)=\{x\in G\,|\,\varphi(x)(1)=1\}$ is a subgroup of $G$ of index $m$. Conversely, if $H$ is a subgroup of $G$ of index $m$, then the action of $G$ on the right 
cosets of $H$ gives rise to a transitive homomorphism $\varphi\colon G\to\Sym_m$ such 
that $H=\mathrm{Stab}_\varphi(1)$. Such $\varphi$ will be called \emph{permutation representation 
associated with $H$}. 

We say that two homomorphisms $\varphi$ and $\psi$ from $G_1$ to $G_2$ are conju­gate if there exists $y\in G_2$ such that $\varphi(x)=y\psi(x)y^{-1}$ for all $x\in G_1$. It 
is easy to see that two subgroups of $G$ are conjugate if and only if the 
associated permutation representations are conjugate. 

For $u\in\Sym_m$ we have the partition of $\{1,\dots, m\}$ into the fixed set 
$F(u)$ and the support $S(u)$ of $u$. We will repeatedly use the fact that 
if $u,v\in\Sym_m$ commute, then $v$ preserves $F(u)$ and $S(u)$. 
\subsection{Mapping class group of a nonorientable surface.}
\begin{figure}
\input{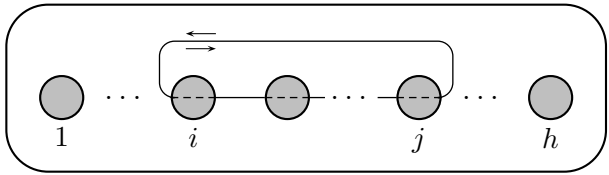}
\caption{\label{gij} The surface $N_{h,1}$ and the curve $\gamma_{i,j}$.}
\end{figure}
Fix $h=2g+r$, where $g\ge 2$, $r\in\{1,2\}$. Let us represent $N_{h,1}$ as a disc with $h$ 
crosscaps. This means that interiors of $h$ small pairwise disjoint discs 
should be removed from the disc, and then antipodal points in each of 
the resulting boundary components should be identified. Let us arrange 
the crosscaps as shown on Figure \ref{gij} and number them from $1$ to $h$. The closed surface $N_{h,0}$ is obtained by gluing a disc along the boundary of $N_{h,1}$.
The inclusion of $N_{h,1}$ in $N_{h,0}$ induces epimorphisms
$\M(N_{h,1})\to\M(N_{h,0})$ and $\T(N_{h,1})\to\T(N_{h,0})$.  
For $i\le j$ let $\gamma_{i,j}$ denote the simple closed curve on $N_{h,1}$ from Figure \ref{gij}. If $j-i$ is odd then  $\gamma_{i,j}$ is two-sided and we will denote by $T_{\gamma_{i,j}}$ the Dehn twist about $\gamma_{i,j}$ in the direction indicated by arrows on Figure \ref{gij}. For $1\le i\le h-1$ set $\alpha_i=\gamma_{i,i+1}$, $\alpha_0=\gamma_{1,4}$ and $\alpha_{2g+2}=\gamma_{1,2g}$ (if $g=2$ then $\alpha_0=\alpha_6$). We can alter the curves $\alpha_i$ by an isotopy, so that they intersect each other minimally. Let $\Sigma$ be a regular neighbourhood of the union of $\alpha_i$ for $1\le i\le 2g$, and let $\Sigma'$ be a regular neighbourhood of the union of $\alpha_i$ for $1\le i\le 2g-2$.  These neighbourhoods may be chosen so that $\Sigma'\subset\Sigma$, $\alpha_{2g+2}\subset\Sigma$, and $\alpha_0\subset\Sigma'$ if $g\ge 3$. Note that $\Sigma$ and $\Sigma'$ are homeomorphic to, and hence will be identified with, respectively $S_{g,1}$ and $S_{g-1,1}$. Figure \ref{Sg1} shows the configuration of the curves $\alpha_i$ on $S_{g,1}$. 
\begin{figure}
\input{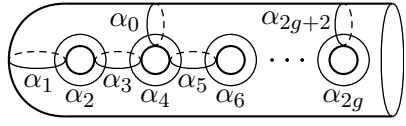}
\caption{\label{Sg1} The surface $S_{g,1}$.}
\end{figure}
Observe that the closure of $N_{h,1}\backslash\Sigma$ is homeomorphic to $N_{r,2}$. By \cite[Cor. 3.8]{Stu_geom} the inclusions $S_{g-1,1}\subset S_{g,1}\subset N_{h,1}$ induce injective homomorphisms 
\[\M(S_{g-1,1})\hookrightarrow\M(S_{g,1})\hookrightarrow\T(N_{h,1}).\] We treat $\M(S_{g-1,1})$ and $\M(S_{g,1})$ as subgroups of $\T(N_{h,1})$.
Set $T_i=T_{\alpha_i}$. It is well known that $\M(S_{g,1})$ is generated by $T_i$ for $0\le i\le 2g$ (originally it was proved by Humphries \cite{Hum} that these $2g+1$ twists generate $\M(S_{g,0})$, if $S_{g,0}$ is obtained by gluing a disc along the boundary of $S_{g,1}$). We define \emph{crosscap transposition} $U$ to be the isotopy class of the homeomorphism interchanging the two rightmost crosscaps as shown on Figure \ref{U}, and equal to the identity outside a disc containing these crosscaps. The composition $T_{h-1}U$ is the Y-homeomorphism defined in \cite{Lick1}. In particular $U\in\M(N_{h,n})\backslash\T(N_{h,n})$.
\begin{figure}
\input{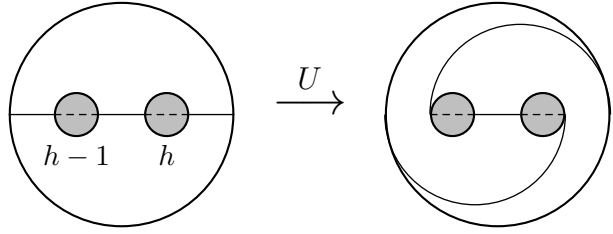}
\caption{\label{U}The crosscap transposition $U$.}
\end{figure}
The next theorem can be deduced from the main result of \cite{PSz}.
\begin{theorem}\label{MCGgens}
For $h\ge 4$ and $n\le 1$, $\M(N_{h,n})$ is generated by $U$ and $T_i$ for $0\le i\le h-1$.
\end{theorem}
For $i\ne j$ we either have $T_iT_j=T_jT_i$ if $\alpha_i\cap\alpha_j=\emptyset$, or 
$T_iT_jT_i=T_jT_iT_j$ if $|\alpha_i\cap\alpha_j|=1$. The last equality is called \emph{braid relation}.
Evidently $U$ commutes with $T_i$ for $1\le i\le h-3$, and if $h\ge 6$ then also with $T_0$. Observe that $U$ preserves (up to isotopy) the curve $\alpha_{h-1}$ and reverses orientation of its regular neighbourhood. Hence
\begin{equation}\label{UTU}
UT_{h-1}U^{-1}=T_{h-1}^{-1}
\end{equation}
There is no similarly short relation between $U$ and $T_{h-2}$. Observe, however, that up to isotopy, $U(\alpha_{h-2})$ intersects $\alpha_{h-2}$ in a single point. Because the local orientation used to define $T_{h-2}$ and that induced by $U$ do not match at the point of intersection, $UT_{h-2}U^{-1}$ satisfies the braid relation with $T_{h-2}^{-1}$. Similarly, if $g=5$ then  $UT_0U^{-1}$ satisfies the braid relation with $T_0^{-1}$.
\begin{cor}\label{Tgens}
For $h\ge 4$ and $n\le 1$, $\T(N_{h,n})$ is generated by $T_i$ for $0\le i\le h-1$,  $U^2$, $UT_{h-2}U^{-1}$, and if $h\le 5$ then also $UT_0U^{-1}$ .
\end{cor}
\begin{proof}
Set $X=\{T_i\}_{0\le i\le h-1}$. By Theorem \ref{MCGgens}, $\M(N_{h,1})$ is generated by $X$ and $U$. Since $\T(N_{h,1})$ is an index 2 subgroup of $\M(N_{h,1})$, $X\subset\T(N_{h,1})$ and $U\in\M(N_{h,1})\backslash\T(N_{h,1})$, thus $\T(N_{h,1})$ is generated by $X$, $UXU^{-1}$ and $U^2$. By the remarks  preceding the corollary we have $UXU^{-1}\subset X\cup X^{-1}\cup\{UT_{h-2}U^{-1}\}$ if $h>5$, and $UXU^{-1}\subset X\cup X^{-1}\cup\{UT_{h-2}U^{-1}, UT_0U^{-1} \}$ if $h\le 5$.
\end{proof}

For a subset $X$ of a group $G$, we denote by $C_GX$ the centraliser of $X$ 
in $G$.
\begin{cor}\label{genCen}
Let $G=\M(N_{h,1})$ or $G=\T(N_{h,1})$ for $h\ge 6$.  Then $G$ is generated by $\M(S_{g,1})\cup C_G\M(S_{g-1,1})$.
\end{cor}
\begin{proof}
Every generator of $G$ from Theorem \ref{MCGgens} or Corollary \ref{Tgens} is either  supported on $S_{g,1}$, or restricts to the identity on $S_{g-1,1}$.
\end{proof}
For  $x,y\in G$ we write $[x,y]=xyx^{-1}y^{-1}$. The commutator subgroup of $G$ is denoted by $[G,G]$, and the abelianization $G/[G, G]$ by $G^\ab$.
\begin{lemma}\label{comm}
Let $h\ge 5$, $n\in\{0,1\}$, $N=N_{h,n}$. Suppose that $T_\alpha$ and $T_\beta$ are Dehn twists about  two-sided simple closed curves  $\alpha$ and $\beta$ on $N$, intersecting at one point. Then $[\M(N),\M(N)]=[\T(N),\T(N)]$ is the normal closure in $\T(N)$ of $[T_\alpha, T_\beta]$. 
\end{lemma}
\begin{proof}
Let $K$ be the normal closure in $\T(N)$ of $[T_\alpha, T_\beta]$. Evidently
$K\subseteq[\T(N),\T(N)]$. Let $F$ be a regular neighbourhood of $\alpha\cup\beta$. Then $F$ is homeomorphic to $S_{1,1}$ and $N\backslash F$ is a nonorientable surface of genus $h-2\ge 3$. It follows that there is a homomorphism $y\colon N\to N$ equal to the identity on $F$ and not isotopic to a product of Dehn twists on $N$ (for example, $y$ may be taken to be a crosscap transposition supported on $N\backslash F$). Now, for $f\in\M(N)\backslash\T(N)$ we have $f[T_\alpha, T_\beta]f^{-1}=(fy)[T_\alpha,T_\beta](fy)^{-1}\in K$. It follows that $K$ is normal in $\M(N)$. By applying \cite[Lemma 3.3]{SzLow} to the canonical projection
$\M(N)\to\M(N)/K$ we have that $M(N)/K$ is abelian, hence $[\M(N),\M(N)]\subseteq K$.
\end{proof}
 The following theorem is proved in \cite{KorkH1} for $n=0$ and generalised to $n>0$  
in \cite{Stu_twist,Stu_bdr}.
\begin{theorem}\label{abel} 
For $h\in\{5, 6\}$ we have $\M(N_{h,n})^\ab=\Z_2\times\Z_2$ and 
 $\T(N_{h,n})^\ab=\Z_2$. For $h\ge 7$ we have $\M(N_{h,n})^\ab=\Z_2$ and 
 $\T(N_{h,n})^\ab=0$.
\end{theorem} 
\section{Some permutation representations of $\M(S_{g,1})$.}\label{Section_orient}
For $g\ge 2$ let $\phi^-_{g,n}\colon\M(S_{g,n})\to\Sym_{m^-_g}$ 
and  $\phi^+_{g,n}\colon\M(S_{g,n})\to\Sym_{m^+_g}$ 
be the representations associated with the subgroups $\cO^-_{g,n}$ and $\cO^+_{g,n}$ 
respectively. The case $g=2$ is special, as $\M(S_{2,n})$ contains another 
subgroup of index $m^-_2=6$, not conjugate to $\cO^-_{2,n}$ , which we will now 
describe (see \cite{BGP} for references for the facts used in this paragraph). 
The group $\Sp(4,\Z_2)=\Sym_6$ has a non­inner automorphism $\alpha$ defined by 
\[\alpha\colon
\begin{cases}
(1~~2) \mapsto (1~~2)(3~~5)(4~~6)\\ 
(2~~3) \mapsto (1~~3)(2~~4)(5~~6)\\ 
(3~~4) \mapsto (1~~2)(3~~6)(4~~5)\\ 
(4~~5) \mapsto (1~~3)(2~~5)(4~~6)\\ 
(5~~6) \mapsto (1~~2)(3~~4)(5~~6) 
\end{cases}\]
It turns out that $\alpha(O^-(4,\Z_2))$ is a subgroup of $\Sp(4,\Z_2)$ of index 
$6$, which is not conjugate to $O^-(4,\Z_2)=\Sym_5$ (on the other hand 
$\alpha(O^+(4,\Z_2))$ is conjugate to $O^+(4,\Z_2)$). Let  $\phi^\alpha_{n}\colon\M(S_{2,n})\to\Sym_6$  be 
the representation associated with the subgroup  $\theta_{2,n}^{-1}(\alpha(O^-(4,\Z_2)))$.
\begin{theorem}[\cite{BGP}]\label{Par}
\begin{enumerate}
\item Suppose that $m\le 10$ and $\phi\colon\M(S_{2,n})\to\Sym_m$ is a nonabelian transitive representation. Then $m\in\{6,10\}$. If $m=6$ then  $\phi$ is  conjugate either to $\phi_{2,n}^-$ or to $\phi_n^\alpha$.  If $m=10$ then  $\phi$ is  conjugate to $\phi_{2,n}^+$.
\item Suppose $g\ge 3$, $m\le m_g^+$ and $m<5m_{g-1}^-$ if $g\ge 4$,  and
$\phi\colon\M(S_{g,n})\to\Sym_m$ is a non-trivial transitive representation.  Then either $m=m_g^-$ and $\phi$ is  conjugate $\phi_{g,n}^-$, or $m=m_g^+$  and $\phi$ is  conjugate to $\phi_{g,n}^+$.
\item For $g\ge 3$ and $n\ge 1$, $\phi_{g,n}^-$ is conjugate to an extension of  $(\phi_{g-1,n}^-)^3\oplus\phi_{g-1,n}^+$ from $\M(S_{g-1,n})$ to $\M(S_{g,n})$, and $\phi_{g,n}^+$ is conjugate to an extension of  $(\phi_{g-1,n}^+)^3\oplus\phi_{g-1,n}^-$ from $\M(S_{g-1,n})$ to $\M(S_{g,n})$.
\item Let $g\ge 3$ and suppose that $T_\alpha$ is a Dehn twist about a  nonseparating simple closed curve on $S_{g,n}$. Then $\phi(T_\alpha)$ is an involution for $\phi\in\{\phi_{g,n}^-,\phi_{g,n}^+\}$. 
\end{enumerate}
\end{theorem}
Implicit in the statement of (3) is the fact that for $n\ge 1$, $\M(S_{g-1,n})$  
naturally embeds in $\M(S_{g,n})$. Such embedding is defined in \cite{BGP}, and 
it is coherent with our identification of $\M(S_{g-1,1})$ as a subgroup of 
$\M(S_{g,1})$.
 
For the proof of the next lemma we need explicit expressions of the 
images of the generators of $\M(S_{2,1})$ under  $\phi_{2,1}^-$,  $\phi_1^\alpha$ and $\phi_{2,1}^+$. (see \cite[Lemma 3.1]{BGP}). 
\[
\phi_{2,1}^-\colon
\begin{cases}
T_1 \mapsto (1~~2)\\ 
T_2 \mapsto (2~~3)\\ 
T_3 \mapsto (3~~4)\\ 
T_4 \mapsto (4~~5)\\ 
T_0 \mapsto (5~~6) 
\end{cases}
\qquad
\phi_1^\alpha\colon
\begin{cases}
T_1 \mapsto (1~~2)(3~~5)(4~~6)\\ 
T_2 \mapsto (1~~3)(2~~4)(5~~6)\\ 
T_3 \mapsto (1~~2)(3~~6)(4~~5)\\ 
T_4 \mapsto (1~~3)(2~~5)(4~~6)\\ 
T_0 \mapsto (1~~2)(3~~4)(5~~6) 
\end{cases}\]
\[
\phi_{2,1}^+\colon
\begin{cases}
T_1 \mapsto (3~~5)(6~~8)(9~~10)\\ 
T_2 \mapsto (2~~3)(4~~6)(7~~9)\\ 
T_3 \mapsto (1~~2)(6~~10)(8~~9)\\ 
T_4 \mapsto (2~~4)(3~~6)(5~~8)\\ 
T_0 \mapsto (4~~7)(6~~9)(8~~10) 
\end{cases}
\]
The following lemma will be used in the next section to prove 
our main result. 
\begin{lemma}\label{main_lemma}
Let $\phi\in\{\phi_{g,1}^-,\phi_{g,1}^+\,|\,g\ge 2\}\cup\{\phi_1^\alpha\}$.  Set
$m=m_g^-$ if $\phi=\phi_{g,1}^-$, $m=m_g^+$ if $\phi=\phi_{g,1}^+$, 
$m=6$ if $\phi=\phi_{1}^\alpha$. For $0\le i\le 2g$ and $i=2g+2$ set
$w_i=\phi(T_i)$. Then
\begin{itemize}
\item[(a)] $C_{\Sym_m}\{w_1,w_2,w_4\}=\{1,w_4\}$ for $g=2$,
\item[(b)] $C_{\Sym_m}\{w_0,w_1,\dots,w_{2g-2},w_{2g-1}\}=\{1,w_{2g+2}\}$ for $g\ge 2$,
\item[(c)] $C_{\Sym_m}\{w_0,w_1,\dots,w_{2g-2},w_{2g}\}=\{1,w_{2g}\}$ for $g\ge 3$,
\item[(d)] $C_{\Sym_m}\phi(\M(S_{g,1}))=\{1\}$ for $g\ge 2$.
\end{itemize}
\end{lemma}
\begin{proof}
The proof is by induction on $g$. For $g=2$ the assertions (a), (b) and (d) may be easily verified by using the expressions for $w_i$ given above. Fix $g\ge 3$ and assume that the lemma is true for $g-1$. By (4) of Theorem \ref{Par}, $w_i$ are involutions.

Note that $w_{2g}\ne w_{2g+2}$. For suppose $w_{2g}=w_{2g+2}$. Then, from
\[w_{2g-1}w_{2g+2}=w_{2g+2}w_{2g-1}\quad\textrm{and}\quad w_{2g}w_{2g-1}w_{2g}=w_{2g-1}w_{2g}w_{2g-1},\] we have $w_{2g-1}=w_{2g}$. By repeating this argument it is easy to show that all $w_i$ are  equal. It follows that the image of $\phi$ is cyclic of order $2$, a contradiction. Since $\phi(\M(S_{g,1}))$ is generated by $\{w_0,w_1,\dots,w_{2g}\}$, (d) follows immediately from (b) and (c).

Let $G$ be the subgroup of $\Sym_m$ generated by $\{w_0,w_1,\dots,w_{2g-2}\}$ and note that $G=\phi(\M(S_{g-1,1}))$. By (3) of Theorem \ref{Par}, the restriction of $\phi_g^-$ (resp. $\phi_g^+$) to $\M(S_{g-1,1})$ is conjugate to $(\phi_{g-1,1}^-)^3\oplus\phi_{g-1,1}^+$ (resp. $\phi_{g-1,1}^-\oplus(\phi_{g-1,1}^+)^3$). It follows that there are three orbits of cardinality $k$ of $G$ and one orbit of cardinality $l$ of $G$, where $k\ne l$. The centraliser $C_{\Sym_m}G$ preserves the $l$-orbit  and permutes the $k$-orbits, which gives a homomorphism $\theta\colon C_{\Sym_m}G\to\Sym_3$. If $u\in C_{\Sym_m}G$ preserves an orbit $X$ of $G$, then the restriction of $u$ to $X$ commutes with the action of $G$, and by (d) of the induction hypothesis, $u$ restricts to the identity on $X$. It follows that $\theta$ is injective.

Set $C_1=C_{\Sym_m}\{w_0,w_1,\dots,w_{2g-2},w_{2g-1}\}$. We have
$C_1\subset C_{\Sym_m}G$ and $w_{2g}\in (C_{\Sym_m}G)\backslash C_1$. Indeed, $w_{2g}$ does not commute with $w_{2g-1}$. For otherwise  we would have $w_{2g}=w_{2g-1}$ by the braid relation, and we would obtain a contradiction as above, by arguing that the image of $\phi$ is cyclic. It follows that $\theta(C_1)$ is a proper subgroup of $\Sym_3$  containing  $\theta(w_{2g+2})$, hence $\theta(C_1)=\{1,\theta(w_{2g+2})\}$ and (b) follows by injectivity  of $\theta$. Similarly, setting $C_2=C_{\Sym_m}\{w_0,w_1,\dots,w_{2g-2},w_{2g}\}$ we have that $\theta(C_2)$ is a proper subgroup of $\Sym_3$, because 
$\theta(w_{2g+2})\notin\theta(C_2)$, hence $\theta(C_2)=\{1,\theta(w_{2g})\}$ and (c) follows. 
\end{proof}
\section{Proofs of the main results} 
In this section we prove Theorem \ref{ThmA}, and the analogous result 
for surfaces of genera $5$ and $6$ which we will now state. Let $h=4+r$ 
for $r\in\{1, 2\}$ and $n\in\{0, 1\}$. 
Set $\widetilde{\cH}_{h,n}^\alpha=\widetilde{\varepsilon}_{h,n}^{-1}(\alpha(O^-(4,\Z_2)))$ and
$\cH_{h,n}^\alpha=\varepsilon_{h,n}^{-1}(\alpha(O^-(4,\Z_2)))$. 
Recall that $\alpha(O^-(4,\Z_2))$ is a subgroup of 
$\Sp(4,\Z_2)$ of index $6$ not conjugate to $O^-(4,\Z_2)$. 
\begin{theorem}\label{ThmB}
Let $h=4+r$ for  $r\in\{1,2\}$ and $n\in\{0,1\}$.
\begin{enumerate}
\item $[\M(N_{h,n}),\M(N_{h,n})]$ is the unique subgroup of $\M(N_{h,n})$ of index $4$ and the unique subgroup of $\T(N_{h,n})$ of index $2$.
\item There are three subgroups of $\M(N_{h,n})$ of index $2$.
\item $\widetilde{\cH}_{h,n}^-$ and $\widetilde{\cH}_{h,n}^\alpha$ (resp. $\cH_{h,n}^-$ and $\cH_{h,n}^\alpha$) are the only subgroups of $\M(N_{h,n})$ (resp.  $\T(N_{h,n})$) of index $6$, up to conjugation.
\item $\widetilde{\cH}_{h,n}^+$ (resp. $\cH_{h,n}^+$) is the unique subgroup of $\M(N_{h,n})$ (resp.  $\T(N_{h,n})$) of index $10$, up to conjugation.
\item All other proper subgroups of $\M(N_{h,n})$ or $\T(N_{h,n})$ have index strictly greater than $10$.
\end{enumerate}
\end{theorem}
Let $G=\M(N_{h,n})$ or $G=\T(N_{h,n})$ for $h\ge 5$ and $n\in\{0,1\}$, 
and $g=\lfloor(h - 1)/2\rfloor$. Suppose that $H$ is a proper subgroup of $G$. If $H$ contains $[G, G]$, then by 
Theorem \ref{abel} either $H=[G, G]$, or $H$ is one of the three subgroups of 
index $2$ of  $\M(N_{h,n})$ for $h\in\{5, 6\}$. Suppose that $H$ does not contain 
$[G, G]$. Then the associated permutation representation $G\to\Sym_m$ is 
non­abelian. To prove Theorem \ref{ThmA} and Theorem \ref{ThmB} it suffices to 
show that there are only two (three if $g = 2$) conjugacy classes of 
non­abelian transitive permutation representations $G\to\Sym_m$  of degree 
$m\le m_g^+$ and $m<5m_{g-1}^-$ if $g\ge 4$. Thus, Theorem \ref{ThmA} and Theorem \ref{ThmB} follow from Theorem \ref{mainT} below and the fact that $\M(N_{h,0})$ (resp. 
$\T(N_{h,0})$) is a quotient of $\M(N_{h,1})$ (resp. $\T(N_{h,1})$).
\begin{theorem}\label{mainT}
Suppose $h=2g+r$ for $g\ge 2$ and $r\in\{1,2\}$, $m\le m_g^+$ and $m<5m_{g-1}^-$ if $g\ge 4$, $G=\M(N_{h,1})$ or $G=\T(N_{h,1})$, and
$\varphi\colon G\to\Sym_m$ is a nonabelian, transitive representation.  Then $m\in\{m_g^-,m_g^+\}$. If $m=m_g^-$ then $\varphi$ is, up to conjugation, the unique extension of $\phi_{g,1}^-$ (or $\phi_1^\alpha$ if $g=2$) from $\M(S_{g,1})$ to $G$. If $m=m_g^+$ then $\varphi$ is, up to conjugation, the unique extension of $\phi_{g,1}^+$ from $\M(S_{g,1})$ to $G$. 
\end{theorem}
For the proof of Theorem \ref{mainT} we need the following lemma. 
\begin{lemma}\label{tran}
Suppose that  $G$, $g$, $r$, $m$ and $\varphi\colon G\to\Sym_m$ are as in Theorem \ref{mainT}. Then $m\in\{m_g^-,m_g^+\}$ and the restriction of $\varphi$ to $\M(S_{g,1})$ is transitive.
\end{lemma}
\begin{proof} 
Let $\varphi'$ be the restriction of $\varphi$ to $\M(S_{g,1})$. The image of $\varphi'$ is not abelian, for otherwise $\varphi$ would be abelian (by Lemma \ref{comm}, $[G,G]$ is normally generated by $[T_1,T_2]$). By Theorem \ref{Par} ((1) and (2)), there is an orbit $X$ of $\varphi(\M(S_{g,1}))$ of order $m_g^-$ or $m_g^+$.
Since $2m_g^->m_g^+$  and $2m_g^->5m_{g-1}^-$, thus $X$ is the unique orbit of $\varphi(\M(S_{g,1}))$ of order at least $m_g^-$. We want to show that $m=|X|$. By transitivity of $\varphi$, it suffices to show that $\varphi(G)$ preserves $X$.

Suppose $g\ge 3$ and $|X|=m_g^-$ (the proof is completely analogous for $|X|=m_g^+$). By (2) of Theorem \ref{Par}, $\varphi(\M(S_{g,1}))$ acts  trivially on the complement of $X$ in $\{1,\dots,m\}$, and the subrepresentation $x\mapsto\varphi(x)|_X$ for $x\in\M(S_{g,1})$ is conjugate to $\phi_{g,1}^-$. By (3) of Theorem \ref{Par}, $X=Y\cup Z_1\cup Z_2\cup Z_3$, where $Y$ is an orbit of $\varphi(\M(S_{g-1,1}))$ of length $m_{g-1}^+$ and $Z_i$ are orbits of $\varphi(\M(S_{g-1,1}))$ of length $m_{g-1}^-$. All other orbits of $\varphi(\M(S_{g-1,1}))$ have length one.  Since  
$\varphi(C_G\M(S_{g-1,1}))$ permutes the orbits of $\varphi(\M(S_{g-1,1}))$, it  preserves $X$. By Corollary \ref{genCen}, $\varphi(G)$ preserves $X$. 

For the rest of the proof assume $g=2$. If $|X|=m_2^+$ then obviously $m=m_2^+$. Suppose that $|X|=m_2^-=6$. Let $X'=\{1,\dots,m\}\backslash X$. By (1) of Theorem \ref{Par}, the action of $\varphi(\M(S_{2,1}))$ on $X'$ is abelian. Since the twists $T_i$ for $0\le i\le 4$ are all conjugate in $\M(S_{2,1})$, they induce the same permutation $\tau$ of $X'$. Since $|X'|\le 4$ and $\M(S_{2,1})^\ab=\mathbb{Z}_{10}$ (see \cite{KorkSurv}), thus $\tau^2=1$.   After a conjugacy in $\Sym_m$ we may suppose that $X=\{1,\dots,6\}$ and 
$\varphi(T_i)=\phi(T_i)\tau$ for $0\le i\le 4$, where $\phi\in\{\phi_{2,1}^-,\phi_1^\alpha\}$ and $\tau\in\{1, (7~~8), (7~~8)(9~~10)\}$. We will use the expressions for $\phi(T_i)$ given in Section \ref{Section_orient}. Set  $w_i=\varphi(T_i)$ for $0\le i\le 3+r$, 
$v=\varphi(U^2)$, and $u=\varphi(U)$ if $G=\M(N_{4+r,1})$. We will repeatedly use the following two easy observations.

\medskip

{\it Observation 1.} Suppose that $a\in\Sym_m$ preserves $S(\tau)$ and  for some $i,j\in\{0,\dots,4\}$ we have $w_ia=aw_i$ and $w_jaw_j=aw_ja$. Then the restriction of $a$ to $S(\tau)$ is equal to $\tau$.

\medskip

{\it Observation 2.} Suppose that $a\in\Sym_m$ preserves $X\cup S(\tau)$ and  for some $i\in\{0,\dots,4\}$ we have $w_iaw_i=aw_ia$. Then $S(a)\subseteq X\cup S(\tau)$.

\medskip

\noindent{\bf Case $r=1$.} Set $w'_3=\varphi(UT_3U^{-1})$, $w_0'=\varphi(UT_0U^{-1})$. 
Observe that, up to isotopy, $U(\alpha_3)$ is disjoint from $T_4(\alpha_3)$ and $\alpha_1$, and it intersects each of the curves $\alpha_3$ and $\alpha_2$ in a single point. Hence $w'_3$ 
commutes with $w_4w_3w_4$ and $w_1$, and satisfies the braid relation with $w_3$ and $w_2$. Similarly, $w'_0$ commutes with $w_1$, $w_2$, $w_3'$ and satisfies the braid relation with $w_0$ and $w_4$. Note, that $v$ and $u$ commute with $w_i$ for $i=1,2$, and by (\ref{UTU}) also with $w_4$.

By Theorem \ref{MCGgens} and Corollary \ref{Tgens}, to prove that $\varphi(G)$ preserves $X$ it suffices to show that $w'_3$, $w'_0$ and $v$ preserve $X$ if $G=\T(N_{5,1})$, and $u$ preserves $X$ if  $G=\M(N_{5,1})$.  

{\bf Subcase 1a: $\phi=\phi_2^-$.} First assume $G=\T(N_{5,1})$.
Since $w'_3$ commutes with $w_1=(1~~2)\tau$ and $w_4w_3w_4=(3~~5)\tau$, it follows easily that $w'_3$ preserves $\{1,2\}$, $\{3,5\}$ and $S(\tau)$. Write $w'_3=v_1v_2v_3v_4$, where $S(v_1)\subseteq\{1,2\}$, $S(v_2)\subseteq\{3,5\}$, $S(v_3)\subseteq S(\tau)$ and $\{1,2,3,5\}\cup S(\tau)\subseteq F(v_4)$. Note that $v_i$ commute with each other for $i\in\{1,\dots,4\}$.
By observation 1 we have $v_3=\tau$. 
Since $w_2$ commutes with $v_4$ (disjoint supports), by the  braid relation $w_3'w_2w_3'=w_2w_3'w_2$ we have $v_4=1$. Similarly, from $w_3w_1=w_1w_3$ and $w_3'w_3w_3'=w_3w_3'w_3$, we have $v_1=1$ and $v_2=(3~~5)$. Hence $w_3'=(3~~5)\tau$. 

Since $w'_0$ commutes with $w_1$, $w_2$, $w_3'$, we have $\{1,2,3,5\}\subseteq F(w'_0)$ and $w'_0$ preserves $S(\tau)$. By observation 1 and braid relations $w_0'w_iw_0'=w_iw_0'w_i$ for $i=0,4$, we have $w_0'=(4~~6)\tau$.

Since $v$ commutes with $w_1$, $w_2$ and $w_4$, we have $\{1,2,3\}\subseteq F(v)$ and $v$ preserves $\{4,5\}$ and $S(\tau)$. We claim that $\{4,5\}\subset F(v)$. For suppose $v(4)=5$. Then $vw_3v^{-1}=(3~~5)\tau'$, where $\tau'=v\tau v^{-1}$. Note that $\tau'$ commutes with $\tau$, and hence $[vw_3v^{-1},w_3']=1$. This implies, by Lemma \ref{comm}, that the image of $\varphi$ is abelian, because $U^2(\alpha_3)$ intersects $U(\alpha_3)$ in a single point (up to isotopy), a contradiction. Hence $v(4)=4$ and $v(5)=5$. We also have $v(6)=6$, for otherwise $vw_0v^{-1}=(5~~i)\tau'$ for $i\notin X\cup S(\tau)$, which contradicts the braid relation between $vw_0v^{-1}$ and $w_0'$. 

Suppose that $G=\M(N_{5,1})$. We have to shown that $u=\varphi(U)$ preserves $X$. Since $u$ commutes with $w_1$, $w_2$ and $w_4$, we have $\{1,2,3\}\subseteq F(u)$ and  $u$ preserves $S(\tau)$. We have
$uw_3u^{-1}=w_3'=(3~~5)\tau$ and $uw_0u^{-1}=w_0'=(4~~6)\tau$. It follows that $u(4)=5$, $u(5)=4$ and $u(6)=6$.

{\bf Subcase 1b: $\phi=\phi_1^\alpha$.}  Because $u$, $v$ and $w_0'$ commute with $w_1$ and $w_2$, they also commute with $w_2w_1=(1~~4~~5)(2~~3~~6)$. It follows that $u$, $v$ and $w_0'$ preserve $S(w_2w_1)=X$. 

Since $w'_3$ commutes with $w_1$, it preserves $S(w_1)=X\cup S(\tau)$. It also commutes with
$w_1w_4w_3w_4=(1~~6)(2~~4)$. By observation 2, it follows that $w'_3$ can be written as $v_1v_2$, where
$v_1$ is one of the permutations $(1~~6)(2~~4)$ or $(1~~2)(4~~6)$ or $(1~~4)(2~~6)$, 
 and $S(v_2)\subseteq\{3,5\}\cup S(\tau)$. If $\tau=1$, then we are done. Suppose that $\tau\ne 1$ and $w'_3$ does not preserve $\{3,5\}$. Then, up to a permutation of $S(\tau)$, we either have $v_2=(3~~7)(5~~8)$, or $v_2=(3~~7)(5~~8)(9~~10)$. 
It can be checked, that for each of the three possibilities for $v_1$, $w_3'$ does not satisfy the braid relation either with $w_2$ or with $w_3$. Hence $v_2=(3~~5)\tau$ and $w_3'$ preserves $X$.

\medskip
 
\noindent{\bf Case $r=2$.} Set  $w'_4=\varphi(UT_4U^{-1})$.  By Theorem \ref{MCGgens} and Corollary \ref{Tgens}, to prove that $\varphi(G)$ preserves $X$ it suffices to show that $w'_4$, $w_5$  and $v$ preserve $X$ if $G=\T(N_{6,1})$, and $w_5$ and $u$ preserve $X$ if  $G=\M(N_{6,1})$.  

{\bf Subcase 2a: $\phi=\phi_2^-$.}  Since $w_5$, $u$ and $v$ commute with $w_1$, $w_2$, $w_3$ and $w_0$, they preserve $X$. 
Furthermore, we have $\{1,2,3,4\}\subseteq F(w_5)$ and it follows easily from observations 1 and 2 and the braid relation $w_5w_4w_5=w_4w_5w_4$ that $w_5=(5~~6)\tau$. 
Observe that, up to isotopy, $U(\alpha_4)$ is disjoint from $T_5(\alpha_4)$. Hence $w'_4$ commutes with $w_5w_4w_5=(4~~6)\tau$. As it also commutes 
 with $w_1$, $w_2$, we have
$\{1,2,3\}\subseteq F(w'_4)$ and $w_4'$ preserves $S(\tau)$ and $\{4,6\}$. 
Let $w'_4=v_1v_2$, where
$S(v_1)\subseteq\{4,6\}\cup S(\tau)$ and $\{1,2,3,4,6\}\cup S(\tau)\subseteq F(v_2)$. 
From $w_3v_2=v_2w_3$ and $w_3w'_4w_3=w'_4w_3w'_4$ we have  $v_2=1$, hence $w'_4$ preserves $X$.

{\bf Subcase 2b: $\phi=\phi_1^\alpha$.}  Since $w_5$, $w'_4$, $v$ and $u$ commute with $w_2w_1$, they preserve $X=S(w_2w_1)$.
\end{proof}

\medskip

\begin{proof}[Proof of Theorem \ref{mainT}.] 
Let $\phi$ be the restriction of $\varphi$ to $\M(S_{g,1})$. By Lemma \ref{tran} and Theorem \ref{Par}, up to conjugation we may assume $\phi\in\{\phi_{g,1}^-,\phi_{g,1}^+\}$ or $\phi=\phi_1^\alpha$ if $g=2$. We will prove that $\varphi$ is determined by $\phi$. 
Set $w_i=\phi(T_i)$ for $0\le i\le 2g$ and $i=2g+2$.

Consider $\varphi^U\colon G\to\Sym_m$ defined by $\varphi^U(x)=\varphi(UxU^{-1})$ for $x\in G$. By Lemma \ref{tran}, the restriction $\phi^U$ of $\varphi^U$ to $\M(S_{g,1})$ is transitive, and by Theorem \ref{Par}, $\phi^U$ is conjugate to $\phi$. Hence, there exist $a\in\Sym_m$, such that $\phi^U(x)=a\phi(x)a^{-1}$ for $x\in\M(S_{g,1})$.

Suppose that $r=1$. If $g\ge 3$, then for $0\le i\le 2g-2$  we have
$\phi^U(T_i)=\phi(T_i)$, and by (\ref{UTU}) and (4) of Theorem \ref{Par}, also 
$\phi^U(T_{2g})=\phi(T^{-1}_{2g})=\phi(T_{2g})$. 
Hence $a\in C_{\Sym_m}\{w_0,w_1,\dots,w_{2g-2},w_{2g}\}$. Similarly, for $g=2$ we have $a\in C_{\Sym_m}\{w_1,w_2,w_4\}$. By (c) and (a) of Lemma \ref{main_lemma} we have $a\in\{1,w_{2g}\}$. We claim that $a=w_{2g}$. For suppose $a=1$. Then $\varphi(UT_{2g-1}U^{-1})=\phi^U(T_{2g-1})=\varphi(T_{2g-1})$. However, since $[UT_{2g-1}U^{-1}, T_{2g-1}]$ normally generates $[G,G]$ by  Lemma \ref{comm}, thus $\varphi$ is abelian, a contradiction. Now suppose that $r=2$. Then we have
$a\in C_{\Sym_m}\{w_0,w_1,\dots,w_{2g-2},w_{2g-1}\}$. By (b) of Lemma \ref{main_lemma} we have $a\in\{1,w_{2g+2}\}$, and by similar argument as above, we obtain $a=w_{2g+2}$. We conclude that $\varphi(UT_{2g-1}U^{-1})=w_{2g}w_{2g-1}w_{2g}$ if $r=1$ and
$\varphi(UT_{2g}U^{-1})=w_{2g+2}w_{2g}w_{2g+2}$ if $r=2$. If $(g,r)=(2,1)$, then 
$\varphi(UT_0U^{-1})=w_4w_0w_4$. 

If $G=\M(N_{2g+r})$ then $\varphi^U(x)=\varphi(U)\varphi(x)\varphi(U)^{-1}$ and by above arguments we obtain $\varphi(U)=w_{2g}$ if $r=1$, and $\varphi(U)=w_{2g+2}$ if $r=2$. In particular $\varphi(U^2)=1$. We claim that $\varphi(U^2)=1$ also for $G=\T(N_{2g+r})$. Set $v=\varphi(U^2)$. Suppose that $r=1$. If $g\ge 3$, then  $v\in C_{\Sym_m}\{w_0,w_1,\dots,w_{2g-2},w_{2g}\}$, and if $g=2$ then $v\in C_{\Sym_m}\{w_1,w_2,w_4\}$. By (c) and (a) of Lemma \ref{main_lemma} we have $v\in\{1,w_{2g}\}$. Suppose $v=w_{2g}$. Then $\varphi(U^2T_{2g-1}U^{-2})=w_{2g}w_{2g-1}w_{2g}=\varphi(UT_{2g-1}U^{-1})$. However, since $[U^2T_{2g-1}U^{-2},UT_{2g-1}U^{-1}]$ normally generates $[G,G]$ by  Lemma \ref{comm}, thus $\varphi$ is abelian, a contradiction. For $r=2$ the argument is similar, using (b) of Lemma \ref{main_lemma}. 

Suppose that $r=2$. Then $\varphi(T_{2g+1})\in C_{\Sym_m}\{w_0,w_1,\dots,w_{2g-2},w_{2g-1}\}$. By (b) of Lemma \ref{main_lemma} we have $\varphi(T_{2g+1})\in\{1,w_{2g+2}\}$. Since $\varphi(T_{2g+1})$ is conjugate to $w_{2g}$, hence it is not trivial, and $\varphi(T_{2g+1})=w_{2g+2}$.

We have shown that the values of $\varphi$ on the generators of $G$ given in Theorem \ref{MCGgens} and Corollary \ref{Tgens} are determined by $\phi$. Hence $\varphi$ is  the unique extension of $\phi$ from $\M(S_{g,1})$ to $G$.
\end{proof}
%

%

\end{document}